\documentclass[a4paper,11pt]{amsart}
\usepackage[T1]{fontenc}
\usepackage{lmodern}
\usepackage[a4paper,margin=28mm]{geometry}
\usepackage{amsmath,amssymb,amsthm,amscd,mathtools,mathrsfs}
\usepackage{enumitem,comment,graphicx,microtype}
\usepackage[all,cmtip]{xy}
\usepackage{tikz-cd}
\usepackage[dvipsnames]{xcolor}
\usepackage{hyperref}
\usepackage{bookmark}

\definecolor{LinkColor}{HTML}{007091} 

\hypersetup{colorlinks=true,linkcolor=LinkColor,citecolor=LinkColor,
 urlcolor=LinkColor,pdfdisplaydoctitle=true,
 pdftitle={Two proofs of the Cassels--Swinnerton-Dyer conjecture for cubic surfaces},
 pdfauthor={Valery Alexeev and Stefan Schreieder}}
\numberwithin{equation}{section}
\newtheorem{theorem}{Theorem}[section]
\newtheorem{lemma}[theorem]{Lemma}
\newtheorem{proposition}[theorem]{Proposition}
\newtheorem{corollary}[theorem]{Corollary}

\theoremstyle{definition}

\theoremstyle{remark}
\newtheorem{remark}[theorem]{Remark}

\newif\ifeditorial
\editorialtrue

\allowdisplaybreaks[1]
\newcommand{\PP}{\mathbb P}
\newcommand{\A}{\mathbb A}

\newcommand{\CP}{\PP}
\newcommand{\OO}{\mathcal O}

\newcommand{\Spec}{\operatorname{Spec}}

\newcommand{\Gal}{\operatorname{Gal}}

\newcommand{\Gr}{\operatorname{Gr}}
\newcommand{\CH}{\operatorname{CH}}

\newcommand{\Frac}{\operatorname{Frac}}

\DeclareMathOperator{\Pf}{Pf}
\DeclareMathOperator{\Ann}{Ann}
\DeclareMathOperator{\diag}{diag}

\DeclareMathOperator{\Tr}{Tr}
\newcommand{\Aff}{\A}

\newcommand{\cI}{\mathcal I}
\newcommand{\cT}{\mathcal T}
\newcommand{\bk}{\overline{k}}

\newcommand{\dashedlongrightarrow}{\xymatrix@1@=15pt{\ar@{-->}[r]&}}
\newcommand{\hooklongrightarrow}{\xymatrix@1@=15pt{\ar@{^(->}[r]&}}
\newcommand{\congpf}{\xymatrix@1@=15pt{\ar[r]^-\sim&}}

\title[The Cassels--Swinnerton-Dyer conjecture]{Two proofs of the Cassels--Swinnerton-Dyer conjecture for cubic surfaces}

\author{Valery Alexeev}
\address{Department of Mathematics, University of Georgia,
Athens, GA 30602, USA}
\email{valery@math.uga.edu}
\author{Stefan Schreieder}
\address{Leibniz University Hannover, Institute of Algebraic Geometry,
Welfengarten 1, 30167 Hannover, Germany}
\email{schreieder@math.uni-hannover.de}

\date{\today}

\begin{document}
\begin{abstract}
The Cassels--Swinnerton-Dyer conjecture asserts that a cubic hypersurface contains a rational point if and only if it contains a point of degree coprime to $3$, or, equivalently, a zero-cycle of degree $1$.
The case of smooth cubic surfaces in characteristic zero has been reduced by Coray (1976) and Voisin (2026) to the case of points of degree $4$.
We give two independent proofs of this missing case and use a lifting argument 
to extend the result to smooth cubic surfaces over arbitrary fields.
We further give a 
separate 
argument for the case of singular cubic surfaces, extending previous work of Coray over perfect fields.
Altogether, this proves the Cassels--Swinnerton-Dyer conjecture for cubic surfaces.
\end{abstract}

\maketitle

\setcounter{tocdepth}{1}
\tableofcontents

\section{Introduction}\label{sec:introduction}

Cassels and Swinnerton-Dyer conjectured that a cubic surface with a point of degree coprime to $3$, or, equivalently, a zero-cycle of degree one, has a rational point; see also the related problems of Segre \cite[p.~2]{segre}. 
The analogous criteria for quadrics and plane cubic curves are due to Springer \cite{springer} and Poincar\'e \cite{poincare}; see \cite[Proposition~2.3]{coray} for a modern treatment of the latter.

Coray formulated the Cassels--Swinnerton-Dyer conjecture for cubics of arbitrary dimension \cite[Conjecture~(CS)]{coray}.
Over perfect fields, he proved the singular surface case \cite[Proposition~8.2]{coray} and reduced the case of smooth cubic surfaces to points of degrees $4$ and $10$ \cite[Theorem~7.1]{coray}.
Further developments and related results appeared in \cite{CoraySingular,CT-coray,CT-SD,BN,CT-zero-cycles,ma,rivera-viray,balestrieri,porzio-thesis}. 

In a recent breakthrough, Voisin \cite{voisin-coray} eliminated
the degree-$10$ case in Coray's reduction: she showed that,
in characteristic zero, a smooth cubic surface with a point
of degree $10$ has a rational point or a point of degree $4$.
This reduces the conjecture for smooth cubic surfaces in
characteristic zero to the following statement, which we
prove in this paper.

\begin{theorem}\label{thm:deg=4-intro}
Let $S$ be a smooth cubic surface over a field $k$ of characteristic zero.
If $S$ has a closed point of degree $4$, then it has a rational point.
\end{theorem}

In Section \ref{sec:specialization}, we combine this result and the aforementioned work of Coray and Voisin with a lifting argument of Ma \cite{ma} to deal with smooth cubic surfaces over arbitrary fields.
Finally, in Section \ref{sec:arbitrary}, we handle singular cubic surfaces over arbitrary fields by reducing to smooth cubic surfaces over $k((t))$, extending Coray's result over perfect fields.
Together, these arguments give the following solution of the Cassels--Swinnerton-Dyer conjecture for surfaces; see Theorem \ref{thm:main} below for a geometric formulation.

\begin{theorem}\label{thm:main:intro}
Let $F\in k[x_0,x_1,x_2,x_3]$ be a nonzero
homogeneous cubic form over a field $k$.
Then $F$ represents zero nontrivially over $k$ if and only if it represents zero nontrivially over a finite field extension $L/k$ of degree coprime to $3$. 
\end{theorem}

By results of Segre, Manin \cite{manin} and Koll\'ar \cite{kollar}, smooth cubic surfaces with a rational point are unirational.
We therefore obtain the following corollary.

\begin{corollary} \label{cor:unirational}
Let $S$ be a smooth cubic surface over a field $k$.
Then $S$ is unirational if and only if it contains a zero-cycle of degree $1$.
\end{corollary}


Another application of Theorem \ref{thm:main:intro} to the Hasse principle for cubic hypersurfaces is given by Colliot-Thélène and Wittenberg in \cite{CTW}.

In view of \cite{coray,ma,voisin-coray}, the key new ingredient in the proof of Theorem \ref{thm:main:intro} is Theorem \ref{thm:deg=4-intro} above.
We give two independent proofs of the latter statement. 
The first proof, presented in Section \ref{sec:degree-five}, uses a reduction of Cassels (see Proposition \ref{prop:four-five} below) to pass from a point of degree $4$ to an effective zero-cycle of degree $5$.
Using work of Buchsbaum--Eisenbud \cite{BE} and Beauville \cite{Beauville}, we show that the surface either has a rational point or admits a linear \(6\times6\) Pfaffian presentation over the ground field. 
The following result then completes the argument, where 
we recall that the Pfaffian $\operatorname{Pf}(A)$ of an alternating
$2n\times 2n$ matrix $A$ is a homogeneous polynomial of degree $n$
in its entries satisfying $\operatorname{Pf}(A)^2=\det(A)$.

\begin{theorem}\label{thm:pfaffian-intro}
Let $k$ be a field of characteristic different from $2$.
Let $A=(a_{ij})$ be an alternating $6\times 6$ matrix of linear forms $a_{ij}\in k[x_0,\dots ,x_3]$ in four variables.
Then the cubic form $\operatorname{Pf} (A)$ given by the Pfaffian of $A$ represents zero nontrivially over $k$.
%
\end{theorem}

The second proof of Theorem \ref{thm:deg=4-intro}, given in Section \ref{sec:degree-four}, works directly with the given closed point $P$ of degree $4$ on $S$. We construct a suitable tangent direction at $P$, giving rise to a nonreduced subscheme $Z$ of $S$ of length $8$ that is contained in the vanishing locus of three linearly independent quadratic forms $Q_1,Q_2,Q_3$.
If the corresponding quadrics have zero-dimensional base locus, then the cubic form $F$ that cuts out $S$ can be written as $F=H_1Q_1+H_2Q_2+H_3Q_3$ for linear forms $H_i$ (see Lemma \ref{lem:net-of-quadrics}), and any $k$-rational point of ${H_1=H_2=H_3=0}$ lies on $S$.
If the base locus of the linear series of quadrics through $Z$ is not zero-dimensional, then we show that either $S$ has a rational point or the base locus is a twisted cubic $\Gamma$.
If $\Gamma$ lies in $S$, then $S$ has a rational point because $\Gamma\cong \CP^1_k$ (see Lemma \ref{lem:twisted-cubic}).
Otherwise, $\Gamma\cap S$ has degree $9$ and contains the subscheme $Z$ of length $8$, so the residual intersection has degree $1$ and yields a rational point.

\begin{remark} \label{rem:delPezzo}
Theorem \ref{thm:main:intro} implies that a smooth del Pezzo surface of degree $3$ over an arbitrary field has a rational point if and only if it has index $1$, equivalently, a zero-cycle of degree $1$.
This equivalence was already known for del Pezzo surfaces of degree at least $4$; the case of perfect fields is due to Coray \cite{coray2}, degree $4$ over arbitrary fields is due to Amer, Brumer and Leep (see e.g.~\cite[Corollary 18.7]{EKM}), and the higher degrees can be deduced from \cite[\S6--\S10]{AB} and the references therein.  
The statement also holds in degree $1$, as every del Pezzo surface of degree $1$ has a rational point, namely the base point of its anticanonical pencil. 
However, the statement fails for del Pezzo surfaces of degree $2$, even in characteristic zero, see \cite{KM17} and \cite[Remarque~4.3]{CT-zero-cycles}.
The analogue of Theorem \ref{thm:main:intro} also fails for quartic curves and surfaces, see \cite[Example 2.8]{coray} and \cite{BN}, respectively.  
\end{remark}

The starting points for this paper were two GPT-assisted proofs obtained independently by the authors; see the AI disclosure below for details of the interactions with GPT and the respective contributions.

\section{Preliminaries and degree reductions} \label{sec:preliminaries}
 
Throughout this section, the field $k$ is perfect.
The degree of a finite $k$-scheme is its length over $k$. 
We use the associated effective zero-cycle when taking residual intersections.

The reductions in Lemmas \ref{lem:secant} and \ref{lem:position} below 
are classical, see \cite[Propositions~2.2 and~2.3]{coray}. 

\begin{lemma}
\label{lem:secant}
Let $X\subset\PP^3_k$ be the zero locus of a nonzero cubic form.
If $X$ has an effective zero-cycle of degree one or two, then
$X(k)\ne\varnothing$.
\end{lemma}
\begin{proof}
It suffices to consider a closed point $P$ of degree two. Its two
conjugates span a $k$-line $\ell$. If $\ell\subset X$, it supplies a
rational point. Otherwise $[\ell\cap X]-[P]$ is an effective
zero-cycle of degree one.
\end{proof}

\begin{lemma}\label{lem:position}
Let $P$ be a closed point of degree $d\in\{4,5\}$ on a smooth cubic
surface $S/k$.
\begin{enumerate}
\item If $P$ is contained in a $k$-plane, then $S(k)\ne\varnothing$.\label{item:lem:position:1}
\item If $P$ is not contained in a $k$-plane, its geometric points are
in general linear position: any four of them span $\PP^3$.\label{item:lem:position:2}
\end{enumerate}
\end{lemma}
\begin{proof}
For \eqref{item:lem:position:1}, which goes back to Poincar\'e \cite{poincare}, 
let $\Pi$ be a $k$-plane that contains $P$. 
Since $h^0(\Pi,\OO_\Pi(2))=6>d$, there is a conic $Q\subset\Pi$
containing $P$. Write $C=S\cap\Pi$. If $C$ and $Q$ have no common
component, then $[C\cap Q]-[P]$ is effective of degree $6-d\le2$,
and Lemma~\ref{lem:secant} applies. Otherwise the greatest common
divisor of their equations over $k$ has degree one or two. In the
first case it defines a $k$-line on $S$; in the second, the residual
factor of the cubic defines such a line.

For \eqref{item:lem:position:2}, the Galois group acts transitively
on the four-element subsets of the geometric points of $P$.
Thus, if one such subset were coplanar, all would be, and the
span of $P$ would be contained in a $k$-plane, a contradiction. 
\end{proof}

The next lemma studies the arithmetic of twisted cubics.

\begin{lemma}\label{lem:twisted-cubic}
Let $\Gamma\subset \PP_k^3$ be a smooth projective  curve over a field $k$, whose base change to $\bar k$ is a twisted cubic.
Then $\Gamma\cong \PP^1_k$ and $\Gamma\subset \PP_k^3$ is a twisted cubic.
\end{lemma}
\begin{proof}
Since $\Gamma_{\bar k}$ is a twisted cubic, $\mathcal O(1)|_{\Gamma}$ is a line bundle of degree three.
Hence, $\omega_\Gamma\otimes \mathcal O(1)$ is a line bundle of degree one and so $\Gamma$ is rational.
Under the identification $\Gamma\cong\PP^1_k$, the embedding $\Gamma\subset \PP_k^3$ must then be given by the full linear series $|\mathcal O_{\PP^1_k}(3)|$, hence $\Gamma$ is a twisted cubic.
\end{proof}

The next proposition follows from a result of Cassels, see
\cite[Proposition~3.1, pp.~274--275]{coray}. 
We give a complete argument for convenience of the reader.

\begin{proposition}[Cassels]\label{prop:four-five}
A smooth cubic surface $S/k$ has an effective zero-cycle of degree
four if and only if it has an effective zero-cycle of degree five.
More precisely, a closed point of either degree yields either a
rational point or a closed point of the other degree.
\end{proposition}
\begin{proof}
If $S(k)\ne\varnothing$, multiples of a rational point give both
cycles. Otherwise Lemma~\ref{lem:secant} excludes points of degrees
one and two; hence an effective zero-cycle of degree
$d\in\{4,5\}$ must be a single closed point $P$. By
Lemma~\ref{lem:position}, its geometric points are in general linear
position.

We claim that $P$ lies on a twisted cubic over $k$.
To see this, write
$E=\kappa(P)=k(\theta)$ for a primitive element $\theta$ and let
\[
 \nu_3:\PP^1_k\longrightarrow\PP^3_k,
 \qquad [s:t]\longmapsto[s^3:s^2t:st^2:t^3].
\]
For $d=4$, choose homogeneous coordinates
$P=[a_0:a_1:a_2:a_3]$ with $a_i\in E$. Noncoplanarity means that
the $a_i$ form a $k$-basis of $E$. Comparing this basis with
$1,\theta,\theta^2,\theta^3$ gives a matrix in $\mathrm{GL}_4(k)$
taking $[1:\theta:\theta^2:\theta^3]$ to $P$.
Its image of $\nu_3(\PP^1)$ is the required curve.

For $d=5$, let $\theta_1,\ldots,\theta_5$ be the conjugates of
$\theta$ in a Galois splitting field. The five points
$[1:\theta_i:\theta_i^2:\theta_i^3]$ for $i=1,\dots ,5$ form a projective frame, by
the Vandermonde determinant. The corresponding conjugates of $P$
also form a projective frame. 
There is a unique projectivity matching
these two labelled frames. Both labellings have the same Galois action, 
so uniqueness makes the projectivity Galois invariant; it therefore
descends to $k$. 
Again its image of $\nu_3(\PP^1)$ is a 
twisted
cubic $\Gamma$ containing $P$.

If $\Gamma\subset S$, the curve supplies a $k$-point. Otherwise
B\'ezout's theorem gives
\begin{equation}\label{eq:residual}
 \deg(S\cap\Gamma)=9,
 \qquad [S\cap\Gamma]-[P]\ \text{effective of degree }9-d.
\end{equation}
Under the standing assumption $S(k)=\varnothing$, this effective
cycle is a single closed point, as above.
\end{proof}


\section{The degree-five argument: Pfaffian presentations}  \label{sec:degree-five}
In Section~\ref{sec:pfaffian} we work over a field of characteristic different from two; then in Section~\ref{sec:five} we restrict ourselves to characteristic zero.
Our goal is to prove Theorem \ref{thm:pfaffian-intro} stated in the introduction and show that it implies the following theorem.

\begin{theorem}\label{thm:five}
A smooth cubic surface over a field of characteristic zero with an
effective zero-cycle of degree five has a rational point.
\end{theorem}

\subsection{Pfaffian cubics have rational points}\label{sec:pfaffian}

Let $V=k^6$ and put $\cT=(\bigwedge^2V^*)^3\simeq\Aff^{45}_k$.
For an alternating form $A\in \bigwedge^2V^*$, a subspace $W\subset V$ is isotropic
if $A|_{W\times W}=0$.

Consider the scheme
    \[
 \cI=\{(W,A_0,A_1,A_2) \; \mid \;  A_i|_{W\times W}=0\ (0\le i\le2)\}
 \subset\Gr(3,V)\times\cT
\]
which parametrizes three-spaces $W\subset V$ that are isotropic for each entry of $(A_0,A_1,A_2)\in \cT$.
The following lemma is an ordered-triple version of Han's incidence calculation
\cite[Proposition~2.4]{Han} (who credits Sasha Kuznetsov).

\begin{lemma}
\label{lem:incidence}
  There is a nonempty open subset $\cT^\circ\subset\cT$, defined over
  $k$, such that 
 the natural map $\mathcal I\times_{\cT}\cT^\circ\to \cT^\circ$ is finite \'etale of degree two. 
\end{lemma}

\begin{proof}
For a three-space $W\subset V$, the restriction map $\bigwedge^2V^*\to\bigwedge^2W^*$ is surjective with $12$-dimensional kernel.
Thus $\cI$ is the total space of a rank-$36$ vector bundle over
$\Gr(3,V)=\Gr(3,6)$.
In particular, $\mathcal I$ is smooth of dimension $45$ and the projection
$\pi:\cI\to\cT$ is proper.

Write $V=V_1\oplus V_2\oplus V_3$, with $\dim V_i=2$, and set
\[
 I\coloneq \begin{pmatrix}1&0\\0&1\end{pmatrix}  ,\qquad J\coloneq \begin{pmatrix}0&1\\-1&0\end{pmatrix} \qquad \text{and}
 \qquad D\coloneq \begin{pmatrix}1&0\\0&2\end{pmatrix}.
\]
Consider the triple
of alternating $6\times 6$ matrices
\begin{equation}\label{eq:triple}
 B_0\coloneq \diag(J,J,J),\qquad B_1\coloneq \diag(J,2J,3J),\qquad
 B_2\coloneq \begin{pmatrix}0&I&I\\-I&0&D\\-I&-D&0\end{pmatrix}.
\end{equation}
We aim to compute the fiber of $\pi$ at $(B_0,B_1,B_2)$ scheme-theoretically.
A $B_0$-isotropic three-space $W$ is Lagrangian 
with respect to the standard symplectic form on $V=k^6$.
Put
$T=B_0^{-1}B_1=\diag(I,2I,3I)$. If $W$ is also
$B_1$-isotropic, then
\[
 B_0(Tw,w')=B_1(w,w')=0\qquad \text{for all }w,w'\in W . 
\]
Hence, $T(W)\subset W^{\perp_{B_0}}=W$.
Therefore, $W$ is $T$-invariant and splits into its intersections with the
three $T$-eigenspaces. Isotropy bounds the dimension of each summand by one,
so $W=\ell_1\oplus\ell_2\oplus\ell_3$ with $\ell_i\subset V_i$
a line. 
This description holds scheme-theoretically.
The common isotropic locus for $B_0,B_1$ is therefore $(\PP^1)^3$.
Writing $\ell_i=[a_i:b_i]$ and imposing isotropy with respect to $B_2$,
we obtain the following description of the fiber under consideration: 
\[
 a_1a_2+b_1b_2=0,\qquad
 a_1a_3+b_1b_3=0,\qquad
 a_2a_3+2b_2b_3=0.
\]
No $a_i$ vanishes at a geometric solution: if $a_1=0$ then
$b_2=b_3=0$, contradicting the third equation; the other cases are
similar. Set $u_i=b_i/a_i$. The 
above equations thus simplify to
\begin{equation}\label{eq:fiber}
 u_2=u_3=-u_1^{-1},\qquad u_1^2+2=0.
\end{equation}
Hence the scheme-theoretic fiber of $\pi$ at $(B_0,B_1,B_2)$ is finite and geometrically reduced of degree two.

At these two geometric points the differential of $\pi$ is an
isomorphism, since source and target are smooth of the same dimension
and the fiber tangent spaces vanish. Thus $\pi$ is \'etale 
there. Properness allows us to remove the image of its non-\'etale
locus and shrink around $(B_0,B_1,B_2)$; the resulting morphism is
proper and quasi-finite, hence finite \'etale, of degree two.
\end{proof}

\begin{proof}[Proof of Theorem~\ref{thm:pfaffian-intro}]
Let $k$ be a field of characteristic different from $2$
and let $F(x)=\Pf(A)$ be the Pfaffian of an alternating $6\times 6$ matrix of linear forms in $x_0,\dots ,x_3$ over $k$.
We may write
\[
 F(x)=\Pf\!\left(\sum_{i=0}^3x_iA_i\right),
\]
where the $A_i$ are alternating $6\times 6$ matrices over $k$. 
The theorem is trivial if $F$ is identical to zero.
We may therefore assume that $S=\{F=0\}\subset \CP^3_k$ is a cubic surface. 

We first deal with the case where 
$(A_0,A_1,A_2)\in\cT^\circ(k)$. Then Lemma~\ref{lem:incidence} gives
an extension $L/k$ of degree at most two and a three-space
$W\subset L^6$ isotropic for $A_0,A_1,A_2$.

The alternating form $A_3|_W$ has a nonzero radical vector $v\in W$,
since $\dim W=3$. Consequently all four linear forms 
\[
 A_0v,\ A_1v,\ A_2v,\ A_3v
\]
belong to the three-dimensional subspace $\Ann(W)\subset(L^6)^*$ 
of linear forms $\lambda\colon L^6\to L$ that vanish on $W$.
Choose a nontrivial relation $\sum c_i (A_i v)=0$ with $c_i\in L$. Since $\bigl(\sum c_i A_i\bigr) v=0$, 
the alternating matrix $\sum c_iA_i$ is 
degenerate, 
so
\[
 F(c_0,c_1,c_2,c_3)^2
   =\det\!\left(\sum_{i=0}^3c_iA_i\right)=0.
\]
This gives an $L$-point of $F=0$. Its image in $S$ has degree at most two,
and Lemma~\ref{lem:secant} (whose proof in fact applies over any field of 
characteristic different from $2$) gives a $k$-point.

For an arbitrary triple $(A_0,A_1,A_2)\in\cT(k)$, 
use the fixed matrices~\eqref{eq:triple} and
put
\[
 A_i(t)=(1-t)A_i+tB_i\quad(0\le i\le2),
 \qquad A_3(t)=A_3,
 \qquad F_t(x)=\Pf\!\left(\sum_{i=0}^3x_iA_i(t)\right).
\]
The generic triple belongs to $\cT^\circ$ because its value at $t=1$
does; moreover $F_t\ne0$ over $k(t)$ since $F_0=F\ne0$.
The preceding argument over $k(t)$ gives a $k(t)$-point of $\{F_t=0\}$.
Scale its homogeneous coordinates into the discrete valuation ring
$k[t]_{(t)}$, with at least one a unit. Reduction modulo $t$ then gives
a $k$-point of $\{F=0\}$.
\end{proof}

\subsection{Five points and the characteristic-zero theorem}\label{sec:five}

\begin{lemma}[Beauville] 
\label{lem:presentation}
Let $Z\subset\PP^3_k$ be a geometrically reduced subscheme of degree
five whose geometric points are in general linear position.
Then every cubic form $F\in k[x_0,\dots,x_3]$ vanishing along $Z$
is the Pfaffian of an alternating $6\times 6$ matrix 
$A=(a_{ij})$ of linear forms $a_{ij}\in k[x_0,\dots ,x_3]$.   
\end{lemma}
 
\begin{proof} 
For a smooth cubic surface (which is the only case needed in this section), this is \cite[Proposition~7.2 and Examples 7.4]{Beauville}. 
The alternative approach suggested in \cite[Remark 7.3(b)]{Beauville}
works for any cubic form. 
For clarity, we provide the proof of this alternative version below. See also
Tanturri \cite[Section~2]{Tanturri}, who gives an explicit algorithm for this Pfaffian presentation starting from five $k$-rational points in general linear position.
  
  Put $S=k[x_0,x_1,x_2,x_3]$, and let $I_Z\subset S$ be the homogeneous
  ideal of $Z$. We claim that the quotient ring $S/I_Z$ is Gorenstein, with graded free resolution
\begin{equation}\label{eq:resolution}
 0\longrightarrow S(-5)\xrightarrow{\ q^t\ } S(-3)^5
  \xrightarrow{\ T\ } S(-2)^5 
 \xrightarrow{\ q\ } S \longrightarrow S/I_Z\longrightarrow 0,
\end{equation}
where $T$ is a $5\times 5$ alternating matrix of linear forms,
\[
 q_i=(-1)^{i+1}\Pf(T_{\widehat i})\qquad(1\le i\le5)
\]
are the signed Pfaffians of the principal submatrices $T_{\widehat i}$ of
$T$, and $q = (q_1,q_2,q_3,q_4,q_5)$.

Let us first prove these statements over an algebraic closure $\bk$. 
The five geometric points can be moved to
$[e_0],[e_1],[e_2],[e_3],[1:1:1:1]$. It is easy to see that the ideal $I_Z\otimes_k\bk \subset S\otimes_k\bk$ is
generated by the following five quadratic polynomials: 
\begin{displaymath}
  \begin{aligned}
    q'_1&=x_0x_2-x_0x_1, &\qquad q'_2&=x_0x_3-x_0x_1,\\
    q'_3&=x_1x_2-x_0x_1, &\qquad q'_4&=x_1x_3-x_0x_1,\\
    q'_5&=x_2x_3-x_0x_1.
  \end{aligned}
\end{displaymath}
These are the signed principal Pfaffians of the alternating matrix
\begin{displaymath}
  T' =
  \begin{pmatrix}
    0 & -x_1 & -x_3 & 0 & x_1 \\
    x_1 & 0 & 0 & x_2 & -x_1 \\
    x_3 & 0 & 0 & x_0 & -x_0 \\
    0 & -x_2 & -x_0 & 0 & x_0 \\
    -x_1 & x_1 & x_0 & -x_0 & 0
  \end{pmatrix}
\end{displaymath}
Put $\bar S=S\otimes_k\bk$ and $q'=(q'_1,\ldots,q'_5)$.
Since the signed principal
Pfaffians of $T'$ generate the height-three ideal $I_Z\bar S$,
the Buchsbaum--Eisenbud structure theorem 
for codimension~$3$ Gorenstein ideals
\cite[Theorem~2.1(1)]{BE} shows that
$\bar S/I_Z\bar S$ is Gorenstein and that the associated
Pfaffian complex \eqref{eq:resolution} (with $q',T'$ in place of $q,T$) is a free resolution. 

Graded Betti numbers and Gorensteinness descend under field extensions.
Thus $S/I_Z$ is Gorenstein and its minimal graded free
resolution has the same ranks and degree shifts.

By the graded version of 
\cite[Theorem~2.1(2) and its proof]{BE} (cf.~the discussion about the graded version on \cite[page 466]{BE}), this resolution
can be chosen to be a Pfaffian resolution as in  \eqref{eq:resolution}.
Its middle map $S(-3)^5\to S(-2)^5$ is therefore represented
by a $5\times5$ alternating matrix $T$ of linear forms
and the quadrics $q_i=(-1)^{i+1}\Pf(T_{\widehat i})$ and
$q=(q_1,\ldots,q_5)$ make \eqref{eq:resolution} exact (over~$k$).

In particular, $I_Z$ is generated by $q_1,\dots q_5$.
If $F$ is a cubic vanishing on $Z$, we may thus write
$F=\sum_{i=1}^{5}\ell_i q_i$ with linear forms $\ell_i$ over $k$.
Expanding along the last column gives
\begin{equation}\tag{3.4}
 F=\operatorname{Pf}\begin{pmatrix}
 T & \ell^t \\
 -\ell & 0
 \end{pmatrix},
 \qquad
 \ell=(\ell_1,\ldots,\ell_5) ,
\end{equation}
as we want.  
\end{proof}

\begin{proof}[Proof of Theorem \ref{thm:five}]
If the cycle has a point of degree one or two in its support, apply
Lemma~\ref{lem:secant}. Otherwise it is a single closed point $P$ of
degree five. The coplanar case is Lemma~\ref{lem:position}\eqref{item:lem:position:1}.
In the remaining case Lemma~\ref{lem:position}\eqref{item:lem:position:2} gives general linear
position, Lemma~\ref{lem:presentation} supplies a Pfaffian presentation
over $k$, and Theorem~\ref{thm:pfaffian-intro} gives a rational point.
\end{proof}

\section{The degree-four argument: nets of quadrics} 
\label{sec:degree-four}
Combining Theorem \ref{thm:five} and Proposition \ref{prop:four-five} proves Theorem \ref{thm:deg=4-intro} stated in the introduction, which assumes $\operatorname{char}(k)=0$.
In this section we give a second proof, which naturally works over any perfect field and yields the following theorem.

\begin{theorem}\label{thm:4->1}
Let $S\subset\PP^{3}_k$ be a smooth cubic surface over a perfect
field. If $S$ has a closed point $P$ of degree four, then $S(k)\ne\varnothing$.
\end{theorem}

\subsection{A rational-point criterion from three quadrics} \label{subsec:3-quadrics-criterion}

Let $S=\{F=0\}\subset \CP^3_k$ be a smooth cubic surface over a field $k$.
If $S$ has a $k$-rational point $P$, then we may after a change of coordinates assume that $P=[1:0:0:0]$, in which case $F=x_1Q_1+x_2Q_2+x_3Q_3$ for quadratic forms $Q_i$.
Conversely, if 
\begin{align}\label{eq:F=HiQi}
F=H_1Q_1+H_2Q_2+H_3Q_3\in k[x_0,x_1,x_2,x_3]
\end{align}
for linear and quadratic forms $H_i\in k[x_0,x_1,x_2,x_3]$ and $Q_i\in k[x_0,x_1,x_2,x_3]$, respectively, then $S$ contains the linear space $\{H_1=H_2=H_3=0\}\subset \CP^3_k$ 
and hence a rational point.

Our strategy is to find a decomposition as in \eqref{eq:F=HiQi}.

\begin{lemma} \label{lem:net-of-quadrics}
Let $S=\{F=0\}\subset \CP^3_k$ be a smooth cubic surface over a field $k$.
Let $Z\subset S$ be a zero-dimensional subscheme of length 8 and assume that there are three quadratic forms $Q_1,Q_2,Q_3\in k[x_0,x_1,x_2,x_3]$ such that $W\coloneq \{Q_1=Q_2=Q_3=0\}\subset \CP^3_k$ has dimension zero and contains $Z$.
Then there are linear forms $H_i\in k[x_0,x_1,x_2,x_3]$ such that \eqref{eq:F=HiQi} holds. 
\end{lemma}
\begin{proof}
Since $W$ has dimension zero and contains $Z$, we must have $W=Z$, because $W$ has length $2^3=8$.
Hence, $\mathcal I_Z=(Q_1,Q_2,Q_3)$ and $Q_1,Q_2,Q_3$ form a regular sequence. 
The Koszul complex therefore yields the exact sequence
\begin{align} \label{eq:koszul}
0 \longrightarrow \mathcal O(-6) \longrightarrow \mathcal O(-4)^{\oplus 3} \longrightarrow \mathcal O(-2)^{\oplus 3} \longrightarrow  \mathcal I_{Z } \longrightarrow 0 
\end{align}
on $\CP^3$. 
Twisting with $\mathcal O(3)$, this yields the exact sequence
\[
0 \longrightarrow \mathcal O(-3) \longrightarrow \mathcal O(-1)^{\oplus 3} \longrightarrow \mathcal O(1)^{\oplus 3} \longrightarrow  \mathcal I_{Z}(3) \longrightarrow 0.
\]
Since $H^{i+1}(\CP^3,\mathcal O(-3))=H^i(\CP^3,\mathcal O(-1))=0$ for $i=0,1$, we conclude that 
$$
H^0(\CP^3,\mathcal O(1)^{\oplus 3} )\stackrel{\cong}\longrightarrow H^0(\CP^3, \mathcal I_{Z}(3)),\qquad (H_1,H_2,H_3)\mapsto \sum_{i=1}^3 H_iQ_i
$$
is an isomorphism.
Finally, since $Z\subset S$, the cubic form $F$ yields a global section of $\mathcal I_{Z}(3)$ and we obtain a decomposition as in \eqref{eq:F=HiQi}.
\end{proof}

\begin{remark} \label{rem:h^0(I_Z)=3}
Twisting \eqref{eq:koszul} with $\mathcal O(2)$ and taking cohomology, one obtains similarly that $H^0(\CP^3, \mathcal I_{Z}(2))\cong H^0(\CP^3,\mathcal O^{\oplus 3} )$ has dimension 3 in the setting of Lemma \ref{lem:net-of-quadrics}. 
\end{remark}

\begin{proposition} \label{prop:existence-of-point}
Let $S=\{F=0\}$ be a smooth cubic surface over a field $k$.
Let $Z\subset S$ be a zero-dimensional subscheme of length 8.
Assume that the base locus of the linear series $|\mathcal I_Z(2)|$ on $\CP^3$ is zero-dimensional.
Then $S$ admits a rational point.
\end{proposition}

\begin{proof}
Since the base locus of $|\mathcal I_Z(2)|$ is zero-dimensional, we can find, after base change to an infinite field, three linearly independent quadrics which have zero-dimensional intersection.
Hence, $h^0(\CP^3_k, \mathcal I_{Z}(2))=h^0(\CP^3_{\bar k}, \mathcal I_{Z_{\bar k}}(2))=3$ by Remark \ref{rem:h^0(I_Z)=3}.
It follows that there are three quadratic forms $Q_1,Q_2,Q_3\in k[x_0,x_1,x_2,x_3]$ which satisfy the assumptions in Lemma \ref{lem:net-of-quadrics}.
Hence, $F$ has a decomposition as in \eqref{eq:F=HiQi} and so $S$ has a rational point.
\end{proof}

A necessary condition for the base locus of $|\mathcal I_Z(2)|$ to be zero-dimensional is $h^0(\CP^3,\mathcal I_Z(2))\geq 3$.
We will study this condition and the resulting base loci in the next subsection; since base loci are compatible with the extension of the ground field, it will be enough to treat the geometric case where the ground field is algebraically closed. 

\subsection{Quadrics through four tangent double points: a geometric analysis} \label{sec:quadrics}

 In this subsection we work over an algebraically closed field $k$.
 Let $S\subset \CP^3_k$ be a smooth cubic surface and let $p_0,\dots ,p_3\in S$ be closed points that do not lie on a plane.
 Up to a coordinate change, we may assume that $p_i=[e_i]$ are the four coordinate vertices in $\CP^3_k$, i.e.~the points represented by the standard basis in $k^4$.
 Let $Z=\sqcup_{i=0}^3Z_i$ be a length 8 subscheme of $S$, where $Z_i$ is of length $2$ and supported at the point $p_i$. 
 
We may represent $Z_i$ by $[e_i+\epsilon v_i]$, where $\epsilon^2=0$ and $v_i=(u_{i0},\ldots,u_{i3})$ is a nontrivial tangent direction at $p_i$ with $u_{ii}=0$.
That is, the line $\{[se_i+tv_i]\; \mid \; [s:t]\in \CP^1\} \subset \CP^3$ is contained in the projective tangent space of $S$ at $p_i$. 
In the affine chart $x_i=1$, with coordinates $y_j=x_j/x_i$ for $j\neq i$, the tangent direction $v_i$ may equivalently be described by
\begin{align} \label{def:uij}
v_i=\sum_{j\neq i}u_{ij}\frac{\partial}{\partial y_j} \in T_{p_i}S.
\end{align}  

\begin{lemma}\label{lem:uij-skew-symmetric}
In the above notation, we have
$
h^0\bigl(\CP^3_k,\mathcal I_{Z}(2)\bigr)\geq 3
$
if and only if there are scalars $\lambda_0,\dots,\lambda_3\in k$,
not all zero, such that
\[
\lambda_i u_{ij}=-\lambda_j u_{ji}\qquad \text{for all $0\leq i<j\leq 3$.}
\]
\end{lemma}

\begin{proof}
A quadric through $p_0,\dots,p_3$ has the form
\begin{equation}\label{def:Q}
Q=\sum_{0\leq i<j\leq 3}q_{ij}x_ix_j.
\end{equation}
The space of such quadrics has dimension $\binom{4}{2}=6$.

For convenience, set $q_{ji}=q_{ij}$ whenever $i<j$. 
The quadric $Q$ in \eqref{def:Q} vanishes along $Z_i$ if and only if
$d_{p_i}Q(v_i)=0$. 
Moreover,
\begin{equation}\label{def:dpQ}
d_{p_i}Q(v_i)=\sum_{j\neq i}q_{ij}u_{ij}.
\end{equation}
Thus the quadrics vanishing along $Z$ form the kernel of the linear map
\[
H^0\bigl(\CP^3,\mathcal I_{Z_{\operatorname{red}}}(2)\bigr)\longrightarrow k^4,\qquad 
Q\mapsto
\bigl(d_{p_0}Q(v_0),\dots,d_{p_3}Q(v_3)\bigr) 
\]
from a six-dimensional vector space to $k^4$.
This kernel has dimension at least $3$ if and only if the four linear functionals $Q\mapsto d_{p_i}Q(v_i)$ are linearly dependent. 
This is equivalent to the existence of scalars $\lambda_i$, not all zero, such that
\begin{equation} \label{def:R}
0=\sum_{i=0}^3\lambda_i d_{p_i}Q(v_i)
 =\sum_{i<j}q_{ij}
   \bigl(\lambda_i u_{ij}+\lambda_j u_{ji}\bigr)
\end{equation}
for every $Q$. Since the coefficients $q_{ij}$ are arbitrary, this is equivalent to
$
\lambda_i u_{ij}=-\lambda_j u_{ji}
$
for every $i<j$, as we want.
\end{proof}

The previous lemma gave a necessary and sufficient condition for $h^0(\CP^3,\mathcal I_Z(2))\geq 3$.
We will concentrate on the case $u_{ij}=-u_{ji}$ for all $i,j$, which ensures $h^0(\CP^3,\mathcal I_Z(2))\geq 3$.
We further assume $u_{ij}\neq 0$ for $i\neq j$ to rule out degenerate cases (which will be dealt with later in Proposition \ref{prop:zero-off-diagonal} via a different argument).
The next lemma characterizes the possible base loci in this situation.

\begin{lemma} \label{lem:base-locus}
Assume in the above notation that $u_{ij}=-u_{ji}$ for all $i,j$ and $u_{ij}\neq 0$ for $i\neq j$.
Then the base locus $B$ of the linear series $|\mathcal I_Z(2)|\subset |\mathcal O_{\CP^3}(2)|$ is one of the following:
\begin{enumerate}
    \item $B=Z$, or
    \item $B$ is a twisted cubic in $\CP^3$.
\end{enumerate}
\end{lemma}
\begin{proof}
Consider the matrix $U\coloneq(u_{ij})\in k^{4\times 4}$ which encodes the tangent directions of $Z$ at the four coordinate vertices.
For convenience, we rename the $u_{ij}$ by $a,b,c,d,e,f$ and get 
\[
  U=\begin{pmatrix}
  0&a&b&c\\
  -a&0&d&e\\
  -b&-d&0&f\\
  -c&-e&-f&0
  \end{pmatrix}, \qquad \text{with}\quad abcdef\ne 0.
\] 
A direct computation using \eqref{def:dpQ} shows that the quadratic forms
\begin{align*}
  Q_1\coloneq \frac da x_0x_1-\frac db x_0x_2+x_1x_2,\\
  Q_2\coloneq \frac ea x_0x_1-\frac ec x_0x_3+x_1x_3,\\
  Q_3\coloneq \frac fb x_0x_2-\frac fc x_0x_3+x_2x_3,
\end{align*}
are linearly independent and satisfy $d_{p_i}Q_j(v_i)=0$ for all $i=0,1,2,3$ and all $j=1,2,3$.
Hence, $Q_1,Q_2,Q_3\in H^0(\CP^3,\mathcal I_Z(2))$ span a three-dimensional subspace.

We aim to describe the base locus $B'\coloneq \{Q_1=Q_2=Q_3=0\}$.
If this is zero-dimensional, then $B$ must be equal to $Z$.
It thus suffices to show that $B'$ is either zero-dimensional or a twisted cubic.
Indeed, if $B'$ is a twisted cubic, then so is $B$, because zero-dimensional $B$ implies $h^0(\CP^3,\mathcal I_Z(2))=3$ by Remark \ref{rem:h^0(I_Z)=3} (cf.~proof of Proposition \ref{prop:existence-of-point}). 

On the plane $x_0=0$, the common zero-set of these quadrics is given by the coordinate vertices $p_1,p_2,p_3$.
It remains to analyze the affine chart $x_0=1$ and we write $x_1=t$.
The vanishing $Q_1=Q_2=0$ gives 
\begin{align} \label{eq:x2-x3}
  x_2=\frac{bdt}{a(d-bt)}\qquad \text{and} \qquad 
  x_3=\frac{cet}{a(e-ct)}.
\end{align}
One checks that neither denominator can vanish at a point of the base locus $B'$.
Substitution into the third equation gives
\[
 Q_3=\frac{t^2\Delta}{a^2(d-bt)(e-ct)}\qquad \text{where}\quad \Delta\coloneq abef-acdf+bcde.
\]
If $\Delta\ne 0$, the base scheme $B'$ is therefore finite.
In fact, the above equation gives a base scheme of length $2$ on the chart $x_0=1$, which is compatible with the fact that we work away from $p_1,p_2,p_3$.

Suppose now that $\Delta=0$. 
Then the vanishing of $Q_3$ is automatic, once \eqref{eq:x2-x3} holds.
Moreover, $be-cd\ne 0$, since otherwise $\Delta=bcde\ne 0$.
We introduce a second variable $s$, homogenize \eqref{eq:x2-x3} with respect to $s$, and clear denominators to find that the base locus is contained in the projective curve described by
\begin{align*}
  x_0&=a s(ds-bt)(es-ct),\\
  x_1&=a t(ds-bt)(es-ct),\\
  x_2&=bdst(es-ct),\\
  x_3&=cest(ds-bt),
\end{align*}
with $[s:t]\in \CP^1$. 
Note that each $x_i$ is homogeneous of degree $3$ in $s,t$, hence may be viewed as an element of $H^0(\CP^1,\mathcal O(3))$.
We check that the coefficient matrix of $x_0,\dots ,x_3$ in the basis $s^3,s^2t,st^2,t^3$ has determinant $ -a^2b^2c^2d^2e^2(be-cd)\ne 0$.
Hence, the above expressions for $x_0,\dots ,x_3$ form a basis of $H^0(\CP^1,\mathcal O(3))$, hence they parametrize a smooth twisted cubic $\Gamma\subset \CP^3$.

To double check our computation, we substitute the above parametrization of $\Gamma$ into $Q_1$ and $Q_2$ and get zero, while substitution into $Q_3$ gives
\[
  s^2t^2(ds-bt)(es-ct)\Delta=0.
\]
Thus all three quadrics vanish on $\Gamma$.
The ideal sheaf of a twisted cubic is always generated by three quadrics (see \cite[Proposition~6.1]{eisenbud-syzygies}) and so we must have $\mathcal I_\Gamma=(Q_1,Q_2,Q_3)$.
This concludes the proof of the lemma.
\end{proof}

\subsection{First arithmetic consequences}

We first show that the assumption $u_{ij}=-u_{ji}$ used in Lemma \ref{lem:base-locus} is automatic in the arithmetic situation, as follows.

\begin{lemma} \label{lem:uij=-uji}
Let $S\subset\mathbb P^3_k$ be a smooth cubic surface over a perfect field $k$, and let $Z\subset S$ be a
subscheme of length 8 supported at a closed point $P$ of degree $4$ whose geometric points span $\CP^3_{\bar k}$.
If $h^0(\mathbb P^3_k,\mathcal I_Z(2))\geq 3$, then, over $\bar k$, and in coordinates such that $Z$ is supported on the coordinate vertices $p_i=[e_i]$, the tangent directions can be chosen so that $u_{ij}=-u_{ji}$ for all $i,j$ in \eqref{def:uij}.
\end{lemma}

\begin{proof}
We work over $\bar k$ and pick coordinates so that the geometric points 
of the support of $Z$ are the four coordinate vertices $p_i=[e_i]$.  
Lemma \ref{lem:uij-skew-symmetric} shows that
\[
R\coloneq \left\{(\lambda_i)\in\bar k^4 \; \mid\; \lambda_i u_{ij}+\lambda_j u_{ji}=0 \text{ for all }i<j\right\}
\]
is nonzero.

Write $f_i(Q)=d_{p_i}Q(v_i)$ for the functionals \eqref{def:dpQ} in Lemma \ref{lem:uij-skew-symmetric}.
By \eqref{def:R}, $R$ is the space of relations $\sum_i\lambda_i f_i=0$.
Since $Z$ is defined over $k$, the natural Galois action on $H^0(\CP^3_{\bar k},\mathcal I_{Z_{\operatorname{red}}}(2))^\vee$ 
satisfies, for all $\sigma\in \Gal(\bar k/k)$,
\[
\sigma(f_i)=c_{\sigma,i}f_{\sigma(i)}
\]
for some $c_{\sigma,i}\in{\bar k}^\times $, where
$\sigma(p_i)=p_{\sigma(i)}$.
Applying $\sigma$ to a relation $\lambda\in R$ therefore induces a semilinear
action on $R$, given explicitly by
\[
(\sigma\cdot\lambda)_{\sigma(i)}
   =c_{\sigma,i}\,\sigma(\lambda_i).
\]
In particular, this action permutes the coordinate hyperplanes
according to the transitive Galois action on the four points.  
Thus, if $R$ were contained in one coordinate hyperplane
$\{\lambda_i=0\}$, it would be contained in all four, contradicting
$R\neq0$.

Since $\bar k$ is infinite, the four proper linear subspaces
$R\cap\{\lambda_i=0\}$ $(i=0,1,2,3)$ cannot cover $R$.
We may therefore choose $(\lambda_i)\in R$ with $\lambda_i\neq0$ for all $i$.
Replacing $v_i$ by $\lambda_i v_i=(\widetilde u_{i0}, \dots, \widetilde u_{i3})$ preserves $Z_i$ and gives $ $
\[
\widetilde u_{ij}=\lambda_i u_{ij}
=-\lambda_j u_{ji}=-\widetilde u_{ji},
\]
as required.
\end{proof}

Lemma \ref{lem:base-locus} describes the base locus in the case where $u_{ij}=-u_{ji}$ is nonzero for $i\neq j$.
The missing case will be handled through the following ad hoc argument.

\begin{proposition}\label{prop:zero-off-diagonal}
Let $S\subset\CP^3_k$ be a smooth cubic surface over a perfect field $k$, and let $P\in S$ be a closed point of degree $4$ whose geometric points span $\CP^3_{\bar k}$.
Let $Z\subset S$ be a subscheme of length $8$ with $Z_{\mathrm{red}}=P$.
Suppose that, after choosing coordinates in which the four geometric
points of $P$ are the coordinate vertices $p_i=[e_i]$, we can write
\[
  Z_{\bar k}=\coprod_{i=0}^3 Z_i,  \qquad Z_i=\{[e_i+\epsilon v_i]\; \mid \; \epsilon^2=0\},  \qquad v_i=\sum_{j=0}^3u_{ij}e_j,
\]
where every $v_i$ is nonzero and $U=(u_{ij})$ is skew-symmetric with $u_{ii}=0$ for all $i$.

If $u_{ij}=0$ for some $i\ne j$, then $S(k)\ne\varnothing$.
\end{proposition}

\begin{proof}
Consider the graph $G$ on the set of vertices $\{0,1,2,3\}$ whose edges are the unordered pairs $\{i,j\}$ with $u_{ij}\ne 0$. 
This graph carries a natural action by the absolute Galois group of $k$. 
The Galois action on the four vertices is transitive, so this graph is regular, i.e.~there is a positive integer $n$ such that its valence at each vertex is $n$. 
Since every $v_i$ is nonzero, while some $u_{ij}$ with $i\neq j$ are zero, we find $n\in \{ 1,2\}$. 

Suppose first that $n=1$. 
After relabeling, the two edges are $\{0,1\}$ and $\{2,3\}$. 
For $i\neq j$, write $L_{ij}=\langle p_i,p_j\rangle$ for the line spanned by $p_i$ and $p_j$. 
Then the subscheme $Z_0\sqcup Z_1$ is contained in $L_{01}$, so the restriction of the cubic equation of $S$ to $L_{01}$ vanishes on a subscheme of length $4$.
Hence, $L_{01}\subset S_{\bar k}$. 
A similar argument shows $L_{23}\subset S_{\bar k}$.

The lines $L_{01}$ and $L_{23}$ have no point in common and their union is Galois invariant. 
In fact, the Galois action interchanges the two lines and so their union is defined over $k$ and there is a quadratic extension $K/k$ over which $L_{01}$ and $L_{23}$ are defined.
Cubic surfaces that contain such skew lines over a quadratic extension are rational, hence contain in particular a rational point.
To construct such a point explicitly, choose $a\in L_{01}(K)$ and a generator $\sigma\in \Gal(K/k)$. 
Then the line $ M=\langle a,\sigma(a)\rangle$ is defined over $k$. 
If $M\subset S$, it supplies a $k$-rational point. 
Otherwise the residual intersection point of $M\cap S$ provides such a point. 

Suppose now that $n=2$. 
Then $G$ is the four-cycle graph $C_4$.
After relabeling, we may assume that
\[
  u_{01}=u_{23}=0,
  \qquad b=u_{02},\quad c=u_{03},\quad d=u_{12},\quad e=u_{13},
  \qquad bcde\ne0.
\]
The three quadrics
\[
  Q_1=x_0x_1,\qquad Q_2=x_2x_3,   \qquad   Q_3=\frac{x_0x_2}{b}-\frac{x_0x_3}{c}      -\frac{x_1x_2}{d}+\frac{x_1x_3}{e}
\]
vanish on $Z_{\bar k}$. 
Indeed, they vanish at the four vertices, and direct substitution using \eqref{def:dpQ} gives $d_{p_i}Q_j(v_i)=0$ for every $i,j$.
Moreover,
\[
  \{Q_1=Q_2=0\}=L_{02}\cup L_{03}\cup L_{12}\cup L_{13}.
\]
The form $Q_3$ restricts nontrivially on each of these lines. 
Thus $\{Q_1=Q_2=Q_3=0\}$ is zero-dimensional. 
In particular, the base locus of $|\mathcal I_{Z_{\bar k}}(2)|$ is zero-dimensional.

Formation of the base locus commutes with extension of the ground field, so the base locus of $|\mathcal I_Z(2)|$ over $k$ is also zero-dimensional and Proposition~\ref{prop:existence-of-point} gives $S(k)\ne\varnothing$.
\end{proof}

\subsection{Construction of the length-8 subscheme}
\label{sec:arithmetic-tangent-double-points}

We now construct a subscheme $Z\subset S$ of length $8$ over the ground field by choosing a suitable tangent direction at a closed point of degree $4$. 

\begin{proposition}\label{prop:existence-of-Z}
Let $S\subset \CP^3_k$ be a smooth cubic surface over a perfect field $k$, and let $P\in S$ be a closed point of degree $4$ whose geometric points span $\CP^3_{\bar k}$.
Then there exists a subscheme $Z\subset S$, defined over $k$ and supported at $P$, such that $Z_{\bar k}$ is a disjoint union of four subschemes of length $2$ and
\[
  h^0(\CP^3_k,\mathcal I_Z(2))\geq 3.
\] 
\end{proposition}

\begin{proof}
Put $L=\kappa (P)$, and let $\mathcal J_{P}\subset \mathcal O_S$ be the ideal sheaf of $P$ in $S$. 
Since $S$ is smooth and $L/k$ is separable, the restriction of the cotangent bundle $\Omega^1_{S}\coloneq \Omega^1_{S/k}$ to $P$,
\[
\Omega^1_{S}|_P\coloneq \Omega^1_{S/k}\otimes_{\mathcal O_S} L\cong \mathcal J_P/\mathcal J_P^2 ,
\]
has dimension $2$ over $L$.
The dual of this space is the tangent space $T_{S}|_P$ at $P$.

The choice of a subscheme $Z\subset S$ of length $8$ which is supported at the point $P$ and whose base change to $\bar k$ splits into 4 subschemes of length $2$ is equivalent to the choice of a line in the $L$-vector space $T_{S}|_P$.
Dually, it corresponds to a one-dimensional quotient of $\Omega^1_{S}|_P$.

Set
\[
  E\coloneq \Omega^1_S(2)|_P   \qquad \text{and}   \qquad V\coloneq H^0(\CP^3_k,\mathcal I_P(2)),
\]
where $\mathcal I_P$ is the ideal sheaf of $P$ in $\CP^3_k$.
The four geometric points of $P$ impose independent conditions on quadrics, since they span $\CP^3_{\bar k}$. 
Thus $\dim_k V=6$.
Moreover, $\dim_L E=2$ and $\dim_k E=8$. 
Restricting a quadric which vanishes at $P$ to $S$ and reducing modulo $\mathcal J_P^2$ defines a canonical $k$-linear map
\[
  \varphi \colon V\longrightarrow H^0\bigl(P,(\mathcal J_P/\mathcal J_P^2)\otimes_L \mathcal O_S(2)|_P\bigr)\cong \Omega^1_S(2)|_P ,
\]
where we use that $ \Omega^1_S(2)|_P\cong (\mathcal J_P/\mathcal J_P^2)\otimes_L \mathcal O_S(2)|_P$.

Let us first express the desired property in terms of this map.
Given any $L$-linear quotient $q\colon E\twoheadrightarrow N$ with $\dim_L N=1$, tensoring with $\mathcal O_S(-2)|_P$ yields a one-dimensional quotient of $\Omega^1_S|_P$, hence a length 8 subscheme $Z_q$ supported at $P$ as above.
A quadric $Q\in V$ vanishes on $Z_q$ precisely when $q(\varphi(Q))=0$.
Therefore
\begin{align}\label{eq:arithmetic-first-order-kernel}
  H^0(\CP^3_k,\mathcal I_{Z_q}(2)) =\ker\bigl(q\circ \varphi\colon V\longrightarrow N\bigr).
\end{align}
Since $\dim_k V=6$ and $\dim_k N=4$, the desired inequality $h^0(\CP^3_k,\mathcal I_{Z_q}(2))\geq 3$ is equivalent to $\operatorname{rank}_k(q\circ \varphi)\leq3$.
Equivalently, there must exist a nonzero $k$-linear functional $\mu\colon N\to k$ such that $\mu\circ q\circ \varphi=0$.
Pulling this functional back along $q$ gives a nonzero $k$-linear functional $\lambda\coloneq \mu\circ q\colon E\to k$ satisfying
\begin{equation}\label{eq:lambda}
  \lambda(\varphi(Q))=0\qquad\text{for every }Q\in V.
\end{equation}
We will first choose such a functional $\lambda$ and then recover an $L$-linear one-dimensional quotient of $E$ through which it factors.

Consider the $k$-vector space
\[
  K=\ker \left( \varphi^\vee \colon \operatorname{Hom}_k(E,k)
  \longrightarrow\operatorname{Hom}_k(V,k)
  \right),
  \qquad \varphi^\vee(\lambda)=\lambda\circ \varphi.
\]
The source has dimension $8$ and the target dimension $6$, so $\dim_k K\geq2$. 
We may therefore choose a nontrivial element $\lambda\in K$.
This gives a $k$-linear functional on $E$ which satisfies \eqref{eq:lambda}.

It remains to show that the $k$-linear map $\lambda$ factors through some one-dimensional $L$-linear quotient $q\colon E\to N\cong L$ of $E$. 
To this end, we use the $k$-linear isomorphism
\begin{equation}\label{eq:arithmetic-trace-duality}
  \begin{aligned}
  \Theta:\operatorname{Hom}_L(E,L)
  &\xrightarrow{\ \sim\ }\operatorname{Hom}_k(E,k),\\
  f&\mapsto \operatorname{Tr}_{L/k}\circ f ,
  \end{aligned}
\end{equation}
provided by the trace map. 
(To see that $\Theta$ is an isomorphism of $k$-vector spaces, it suffices to note that it is $k$-linear and injective, while source and target have the same finite dimension over $k$.
To see injectivity, note that any nonzero $L$-linear map $f\colon E\to L$ is surjective and $\operatorname{Tr}_{L/k}$ is surjective because $L/k$ is separable, hence $\operatorname{Tr}_{L/k}\circ f$ is surjective and therefore nonzero as long as $f$ is nonzero.)

It follows that there is a unique nonzero $L$-linear map $f\colon E\to L$ such that $\lambda=\operatorname{Tr}_{L/k}\circ f$.
Since $\lambda$ is nonzero, $f$ must be surjective, and $\lambda$ factors through it; the role of $\mu\colon N\cong L\to k$ above is then taken by the trace map $\operatorname{Tr}_{L/k}$.
As explained above, the length $8$ subscheme that is associated to the one-dimensional quotient $f$ has then the properties asserted in the proposition.
This concludes the proof.
\end{proof}

\subsection{Proof of Theorem \ref{thm:4->1}}

\begin{proof}[Proof of Theorem \ref{thm:4->1}]
By Lemma \ref{lem:position}, we may assume that the geometric points associated to $P$ span $\PP^3_{\bar k}$.
By Proposition \ref{prop:existence-of-Z}, there is a closed subscheme $Z\subset S$ of length $8$ supported on $P$ such that $Z_{\bar k}$ is a disjoint union of four subschemes of length $2$ and such that $h^0(\CP^3,\mathcal I_Z(2))\geq 3$.

We pass to the algebraic closure $\bar k$ of $k$ and choose coordinates on $\CP^3_{\bar k}$ such that the four geometric points $p_0,\dots ,p_3$ that correspond to $P$ agree with the four coordinate vertices: $p_i=[e_i]$.
We may then write $Z_{\bar k}=\sqcup_{i=0}^3Z_i$, where $Z_i$ is the length $2$ subscheme supported at $p_i$ that is given by $[e_i+\epsilon v_i]$ with $\epsilon^2=0$ and $v_i=(u_{i0},\dots,u_{i3})$ with $u_{ii}=0$.
By Lemma \ref{lem:uij=-uji}, we can arrange that $u_{ij}=-u_{ji}$.
By Proposition \ref{prop:zero-off-diagonal}, we may further assume that $u_{ij}$ is nonzero for all $i\neq j$.
It then follows from Lemma \ref{lem:base-locus} that the base locus of $|\mathcal I_{Z_{\bar k}}(2)|$ is either finite or a twisted cubic.
It follows that the base locus of the linear series $|\mathcal I_{Z}(2)|$ on $\CP^3_k$ over the ground field $k$ is either finite or a smooth curve $\Gamma$ whose base change to $\bar k$ is a twisted cubic, hence $\Gamma\cong \CP^1_k$ is a twisted cubic by Lemma \ref{lem:twisted-cubic}.

If the base locus of $|\mathcal I_Z(2)|$ is zero-dimensional, then $S$ has a rational point by Proposition \ref{prop:existence-of-point}.
Otherwise, the base locus is a twisted cubic $\Gamma$. 
If $\Gamma\subset S$, then $S$ has a rational point because $\Gamma\cong \CP_k^1$.
Otherwise, $\Gamma\cap S$ is a zero-dimensional subscheme of length $9$ which contains the length $8$ subscheme $Z$; the residual intersection point then yields a rational point.
This concludes the proof of the theorem.
\end{proof} 

\section{From characteristic zero to positive characteristic}\label{sec:specialization}

We begin with the following.

\begin{theorem}\label{thm:charzero}
A smooth cubic surface over a field of characteristic zero with a
zero-cycle of degree one has a rational point.
\end{theorem}
\begin{proof}
Voisin's theorem \cite[Theorem~1.5]{voisin-coray}, which improved 
Coray's reduction \cite[Theorem 7.1]{coray},
gives a point of degree one or four. 
We may then either apply Theorem \ref{thm:4->1}, or
use Proposition~\ref{prop:four-five} to produce 
an effective zero-cycle of degree five and apply 
Theorem~\ref{thm:five}.
\end{proof}

In the remainder of this section, we reduce the case 
of arbitrary base fields 
(including imperfect ones) to the case of characteristic zero.
The required lifting statement follows 
\cite[Section~2.2, Lemmas~2.4--2.5]{ma}; we include its proof.

\begin{lemma}\label{lem:lift}
Let $k$ have characteristic $p>0$, let $S\subset\PP^3_k$ be a smooth
cubic surface, and let $P\subset S$ be a closed point of degree $d$.
There exist a complete discrete valuation ring $R$ of characteristic
zero with residue field $k$, a smooth cubic surface
$\mathscr S\subset\PP^3_R$ with special fiber $S$, and a finite flat
subscheme $\mathscr Z\subset\mathscr S$ of degree $d$ with special
fiber $P$.
\end{lemma}
\begin{proof}
Choose a Cohen ring $R$ with uniformizer $p$ and residue field exactly
$k$ \cite[Tag~0328]{Stacks}. Lift the coefficients of a defining cubic
of $S$ to obtain $\mathscr S\subset\PP^3_R$. This hypersurface is
flat over $R$ and smooth along its special fiber. Its nonsmooth locus
is closed and proper over $R$. If nonempty, its image would contain
the closed point of $\Spec R$, a contradiction. Thus
$\mathscr S/R$ is smooth.

Since $R$ is a discrete valuation ring and $\mathcal S/R$ is smooth
and projective of relative dimension two, Fogarty's theorem
\cite[Theorem~2.9]{Fogarty} implies that
$\operatorname{Hilb}^d(\mathcal S/R)$ is smooth over $R$, see also \cite[Tag~06D9]{Stacks} and \cite[Lemma~2.5]{ma}.
As $R$ is complete, hence henselian, the $k$-point $[P]$ lifts
to an $R$-point. The corresponding universal family gives the
required finite flat subscheme $\mathcal Z\subset\mathcal S$. 
\end{proof}

We are now ready to prove the case of smooth cubic surfaces.

\begin{theorem}
\label{thm:smooth}
Let $S\subset\PP^3_k$ be a smooth cubic surface over an arbitrary
field. If $S$ has a zero-cycle of degree one, then
$S(k)\ne\varnothing$.
\end{theorem}
\begin{proof}
By Theorem~\ref{thm:charzero}, only positive characteristic remains.
Choose a closed point $P$ of degree $d$ prime to three from the support
of a zero-cycle of degree one. Apply Lemma~\ref{lem:lift}, and put
$K=\Frac(R)$. The finite scheme $\mathscr Z_K$ has degree $d$,
so its fundamental cycle, together with the degree-three
intersection class with a line, gives a zero-cycle of degree one on $\mathscr S_K$.

Theorem~\ref{thm:charzero} gives a $K$-point of $\mathscr S_K$.
Properness extends it to an $R$-point of $\mathscr S$, whose reduction
is a $k$-point of $S$.
\end{proof}

\section{Removing the smoothness hypothesis}\label{sec:arbitrary}

In this section we finish the proof of the following theorem, which is equivalent to Theorem \ref{thm:main:intro}.

\begin{theorem}\label{thm:main}
Let $X=\{F=0\}\subset\PP^3_k$ be a cubic surface over a field $k$.
Then $X$ has a $k$-rational point if and only if it has a closed point of degree coprime to $3$, or, equivalently, a zero-cycle of degree one.
\end{theorem}

For singular cubic surfaces over perfect fields, the conclusion of
Theorem~\ref{thm:main} is due to Coray
\cite[Proposition~8.2]{coray}; see also
\cite[Corollary~1]{CoraySingular} for a second proof. We give a
reduction to the smooth case over a Laurent-series field that applies
over arbitrary fields.  

Let $X=\{F=0\}\subset\PP^3_k$ for a nonzero homogeneous cubic $F$,
with no assumption on $k$ or on the singularities of $X$. If $F$ is
reducible over $k$, it has a linear factor $\ell$, so $X$ contains
the $k$-plane $\{\ell=0\}$ and $X(k)\ne\varnothing$. Write
$X_{\mathrm{sm}}$ for the smooth locus over $k$. We first handle the
possibility that a zero-cycle of degree one is supported entirely in
the singular locus.

\begin{lemma}\label{lem:singular-point}
Let $L/k$ be a finite extension with $3\nmid[L:k]$. If
$X(L)\ne\varnothing$, then either $X(k)\ne\varnothing$ or
$X_{\mathrm{sm}}(L)\ne\varnothing$.
\end{lemma}
\begin{proof}
Choose $p\in X(L)$. Its multiplicity on the hypersurface $X_L$ is
one, two, or three. Multiplicity one means that $p$ is smooth. 
If $p$ has multiplicity two, we will construct a smooth $L$-rational point;
if it has degree $3$ we will show that $X$ has a $k$-rational point.

Suppose the multiplicity is two. In coordinates over $L$ with
$p=[1:0:0:0]$, the equation is
\[
 F=x_0Q(x_1,x_2,x_3)+C(x_1,x_2,x_3),
 \qquad Q\ne0,
\]
where $Q$ and $C$ have degrees two and three, respectively. There exists
$a\in L^3$ with $Q(a)\ne0$, even over a finite field, because the values
$Q(e_i)$ and $Q(e_i+e_j)-Q(e_i)-Q(e_j)$ recover all coefficients of
$Q$.  
Then the point
\[
 q=\left[-\frac{C(a)}{Q(a)}:a_1:a_2:a_3\right]
\]
lies on $X_L$ and satisfies $\partial F/\partial x_0(q)=Q(a)\ne0$.
It is therefore a smooth point of $X_L$,  and so $X_{\rm sm}(L)\neq \emptyset$. 

Suppose now that $p$ has multiplicity three. 
In coordinates with $p=[1:0:0:0]$, the equation $F$ 
that cuts out $X$ is then independent of $x_0$. In the original
coordinates, a nonzero vector $v\in L^4$ representing $p$ consequently
satisfies
\begin{equation}\label{eq:translation}
 F(x+s v)=F(x)
 \quad\text{in }L[x_0,x_1,x_2,x_3,s].
\end{equation}
If $\operatorname{char}k\ne3$, consider the $k$-linear map
\[
 D_F \colon k^4\longrightarrow k[x_0,x_1,x_2,x_3]_2,
 \qquad w\longmapsto D_wF=\sum_{i=0}^3w_i\frac{\partial F}{\partial x_i}.
\]
By differentiating \eqref{eq:translation} with respect to $s$,
we obtain $D_vF=0$. Thus $D_F\otimes_k L$ has nonzero kernel,
and hence so does $D_F$.
We may thus pick a nonzero vector $w\in k^4$ in the kernel of 
the linear map $D_F$. Euler's identity gives
$0=(D_wF)(w)=3F(w)$, so $[w]\in X(k)$.

If $\operatorname{char}k=3$, the hypothesis on $[L:k]$ implies that
$L/k$ is separable. Normalize $v$ so that 
its $i_0$-th coordinate is $1$ for some fixed $i_0$. 
Every conjugate $\sigma(v)$ satisfies~\eqref{eq:translation}, and
directions satisfying this identity are closed under addition, since
translations compose. Thus
\[
 w\coloneq \Tr_{L/k}(v)=\sum_{\sigma:L\hookrightarrow\bk}\sigma(v)\in k^4
\]
also satisfies~\eqref{eq:translation}. 
By our normalization, the $i_0$-th coordinate of $w$ is
$[L:k]\ne0$ in $k$, so $w\ne0$.
Setting $x=0$ and $s=1$ in \eqref{eq:translation} then gives
$F(w)=0$. This concludes the argument. 
\end{proof}

\begin{proposition}
\label{prop:smoothing}
Suppose every smooth cubic surface over $k((t))$ with a zero-cycle of
degree one has a $k((t))$-rational point. Then every cubic surface over
$k$ with a zero-cycle of degree one has a $k$-rational point.
\end{proposition}
\begin{proof}
Let $Y=\{F=0\}$ be a cubic surface over $k$ with a zero-cycle of degree one.
Choose a closed point of degree $d$ prime to three from the support of
the given zero-cycle, and let $L/k$ be its residue field extension.
By Lemma~\ref{lem:singular-point}, we may assume that there is a point
$q\in Y_{\mathrm{sm}}(L)$.

Choose a homogeneous cubic $G$ defining a smooth surface over $k$.
For example, one may take
\[
 G\coloneq
 \begin{cases}
 x_0^3+x_1^3+x_2^3+x_3^3,
     &\operatorname{char}k\ne3,\\
 x_0^3+x_0x_1^2+x_1x_2^2+x_2x_3^2,
     &\operatorname{char}k=3.
 \end{cases}
\]

Put $R=k[[t]]$, $K=k((t))$, and
\[
 \mathscr X\coloneq
 \{(1-t)F+tG=0\}\subset\PP^3_R.
\]
This is a flat projective family with special fiber $Y$ and smooth
generic fiber $X\coloneq\mathscr X_K$. Indeed, the pencil over
$\Aff^1_k$ contains the smooth member $\{G=0\}$ at $t=1$, so its
generic fiber is smooth.

The point $q$ lifts to a point
$\tilde q\in\mathscr X(L[[t]])$ by Hensel's lemma: on an affine chart
containing $q$, choose a coordinate for which the partial derivative
of the defining equation is nonzero at $q$. Fixing the other
coordinates at their values at $q$, Hensel's lemma lifts the remaining
coordinate to a power series in $t$ satisfying the equation of
$\mathscr X$; see also \cite[Lemma~2.4]{ma}.

Put $M=L((t))$. Restricting $\tilde q$ to the generic point of
$\operatorname{Spec}L[[t]]$ gives a point $\tilde q_\eta\in X(M)$.
A $k$-basis of $L$ is also a $K$-basis of $M$, so $[M:K]=d$.
Let
\[
 \pi\colon X_M=X\times_K M\longrightarrow X
\]
be the projection, and regard $\tilde q_\eta$ as an $M$-rational
point of $X_M$. If $P$ denotes its image in $X$, then
\[
 \zeta\coloneq\pi_*[\tilde q_\eta]
   =[M:\kappa(P)]\cdot[P]\in\CH_0(X),
 \qquad \deg_K\zeta=d.
\]
Since $\gcd(d,3)=1$, an integral linear combination of $\zeta$ and
the degree-three intersection class with a line gives a zero-cycle
of degree one on $X$.

Since $X$ is smooth and has a zero-cycle of degree one, the hypothesis
of the proposition implies that $X(K)\ne\varnothing$.
By the valuative criterion of properness, any point of $X(K)$ extends
to an $R$-point of $\mathscr X$, whose reduction belongs to $Y(k)$.
\end{proof}

\begin{proof}[Proof of Theorem~\ref{thm:main}]
Apply Proposition~\ref{prop:smoothing}, using Theorem~\ref{thm:smooth}
over $k((t))$.
\end{proof}

\section*{AI disclosure} 
The starting points for this paper were two different 
GPT-assisted proofs of the Cassels--Swinnerton-Dyer conjecture for cubic surfaces, obtained independently by the authors. The paper is the synthesis of these arguments, with GPT used to help with the copyediting.
The authors’ interactions with GPT and their respective 
contributions were as follows. 

VA asked GPT-5.5 to find a 
counterexample to the conjecture,  
suggesting many ideas for it to try,
including the universal family
of four-pointed cubic surfaces as a candidate. 
GPT-5.5 made no progress, but when VA put the same question to 
GPT-6 Astra, it produced a rational section of this family. 
While the initial argument seemed convoluted, further exchanges 
with GPT led to the key idea presented 
in Section \ref{sec:degree-five}.
At VA’s request, GPT then developed extensions to arbitrary 
fields and arbitrary cubic forms, leading to the arguments 
presented in Sections \ref{sec:specialization} and \ref{sec:arbitrary}.  

Having set aside the search for counterexamples some time 
earlier, SS explored various approaches to constructing a 
rational point from a point of degree $4$ in conversations with 
GPT. After exploring one such approach, involving quadrics through effective 
zero-cycles of degree $8$ (cf.~Section \ref{subsec:3-quadrics-criterion}), 
GPT pointed out that the dimension 
of the space of quadrics through a nonreduced subscheme of 
length $8$ supported at four general geometric points can jump for 
special choices of the tangent directions, 
cf.~Section \ref{sec:quadrics}. It suggested 
that this observation offered a promising approach, 
while noting that the details still needed to be worked 
out and checked. SS developed and checked these details 
with GPT’s assistance, leading to the results presented 
in Sections \ref{sec:degree-four} and \ref{sec:specialization}.

\section*{Acknowledgments} 
SS first learned about this problem from Johan de Jong in 2020 in connection with the work of his student Qixiao Ma, who subsequently held a postdoctoral position in SS's group from 2021 to 2023.  
VA learned about the problem from a talk given by Claire Voisin at a conference in Princeton in June of 2026.

We thank Jean-Louis Colliot-Th\'el\`ene, Philip Engel, Qixiao Ma, Matthias Paulsen, Alexei Skorobogatov and Claire Voisin for discussions and comments. 

VA was partially supported by NSF grant DMS-2501855. VA's free subscription to GPT was provided by the OpenAI for Academic Researchers program. 

This project has also received funding from the European Research Council (ERC) under the European Union’s Horizon 2020 research and innovation programme under grant agreement No 948066 (ERC-StG RationAlgic).
The research was partly conducted in the framework of the DFG-funded research training group RTG 2965: From Geometry to Numbers, Project number 512730679.

\end{document}